\documentclass[final,3p,times]{elsarticle}

\usepackage{amssymb}
\usepackage{amsmath}
\usepackage{amsthm}

\newtheorem{theorem}{Theorem}
\newtheorem{lemma}[theorem]{Lemma}
\newtheorem{cor}[theorem]{Corollary}
\newtheorem{remark}[theorem]{Remark}
\usepackage{mathtools}
\mathtoolsset{showonlyrefs}
\usepackage{physics}
\usepackage[colorlinks=true]{hyperref}

\journal{Journal xxx}

\begin{document}

\begin{frontmatter}



  \title{Well-posedness of the weakly singular Burton--Miller equation for Helmholtz transmission problems}


  \author{Yasuhiro Matsumoto} 

  \affiliation{organization={Center for Information Infrastructure, Institute of Science Tokyo},
    addressline={2-12-1-I8-21, Ookayama, Meguro-ku},
    city={Tokyo},
    postcode={152-8550},
    country={Japan}}

  \begin{abstract}
    Although various boundary integral formulations are available for the Helmholtz transmission problem,
    the weakly singular Burton--Miller equation is promising because it possesses favorable properties for numerical computation.
    This paper rigorously shows that the weakly singular Burton--Miller equation is well-posed.
    We establish bounded invertibility on $C(\Gamma)\times C(\Gamma)$ for $C^2$ boundaries by proving injectivity and applying the Fredholm alternative.
    For the prescribed analytic incident field, the integral equation further yields $u\in C^{1,\beta}(\Gamma)$ and $q\in C^{0,\beta}(\Gamma)$, for $0<\beta<1$.
    We also provide a proof that its fictitious eigenvalues are associated with an exterior Dirichlet problem and an interior impedance problem.
  \end{abstract}




\end{frontmatter}



\section{Introduction}
The Helmholtz transmission problem \cite{kress1977transmission} is important in both physics and engineering applications because it is associated with scattering problems.
When handling an unbounded domain, boundary integral formulations are promising since their solutions naturally satisfy the radiation condition.
Furthermore, boundary integral equations (BIEs) can reduce the spatial dimension by one.
These features enable the development of well-founded and efficient numerical methods.

Various BIEs exist for the same transmission problem.
A typical example is the PMCHWT (Poggio--Miller--Chang--Harrington--Wu--Tsai) equation \cite{chew2001fast},
which is derived based on Green's representation theorem and the boundary conditions of the transmission problem.
This equation involves a boundary layer operator with a \mbox{hypersingular} kernel, which makes its numerical treatment difficult.
On the other hand, the Müller equation \cite{muller1969foundations} can be constructed by utilizing the constraint in BIEs that integral representations of direct potentials are zero within the complementary domain.
Since the Müller equation contains at most weakly singular kernels,
it is frequently employed to develop numerical methods based on the Nystr\"om discretization.
A variant of the Müller equation using indirect potentials is also known \cite{ROKHLIN1983257}.
However, formulations of BIEs often suffer from fictitious eigenvalues that are unrelated to the original boundary value problem \cite{CHEN1998529}.
Although the aforementioned formulations ensure the existence of a unique solution at positive real frequencies,
quasi-mode resonances, which correspond to eigenvalues with a small imaginary part \cite{stefanov2000resonances},
adversely affect the accuracy of numerical methods, even at positive real frequencies \cite{misawa2017boundary}.
To avoid numerical breakdown due to quasi-mode resonances, augmented formulations for the PMCHWT and Müller equations have been proposed \cite{hiptmair2022spurious}.
However, solving these augmented equations is not necessarily computationally efficient.

It has been reported that the PMCHWT and Müller equations possess fictitious eigenvalues associated with transmission problems that differ from the original;
these eigenvalues induce quasi-mode resonances even though the original transmission problem does not exhibit them \cite{misawa2017boundary}.
Therefore, the use of a single integral equation \cite{kleinman1988single} or a Burton–Miller (BM) equation \cite{burton1971application} is appealing, as their fictitious eigenvalues are unrelated to transmission problems.
However, in the single integral equation, dealing with products of boundary layer operators is not straightforward when using Nyström methods.
It should also be noted that the conventional BM equation for transmission problems involves a hypersingular operator.

This paper focuses on a variant of the BM-type equation that involves at most weakly singular kernels,
referred to as the weakly singular BM equation, and investigates its well-posedness.
This equation is expected to be well-suited for numerical computation because its fictitious eigenvalues are unrelated to  transmission problems, and it does not involve kernels more singular than weakly singular ones.
This formulation is equivalent to that in \cite{MISAWA201912-191201}.
That prior study pointed out the boundary value problems associated with fictitious eigenvalues and numerically investigated the eigenvalue distribution using the Sakurai--Sugiura method \cite{asakura2009numerical} after Nystr\"om discretization based on the Alpert quadrature \cite{alpert1999hybrid}.
Regarding the validity of the formulation, however, the uniqueness of the solution was addressed only qualitatively.
We note that while the Burton--Miller BIE is typically applied to Dirichlet or Neumann problems,
its application to transmission problems has not been extensively investigated.
Another study investigates an MPI/OpenMP parallel implementation for the weakly singular BM equation using a Nystr\"om method with the zeta-corrected quadrature \cite{matsumoto2026proxy}.
This paper provides a rigorous proof of
bounded invertibility
of the weakly singular BM equation 
on $C(\Gamma) \times C(\Gamma)$ and establishes the regularity of the solution for the prescribed analytic incident field.

\section{Transmission problems} \label{sec:problem}
Let $\Omega_0$ and $\Omega_1$ be open unbounded and bounded subsets of $\mathbb R^d$ for $d = 2, 3$, respectively, such that
$\overline \Omega_0 \cup \Omega_1 = \Omega_0 \cup \overline \Omega_1 = \mathbb{R}^d$ and $\Omega_0 \cap \Omega_1 = \emptyset$.
Assume $\Omega_0$ is connected.
Let $\Gamma = \partial \Omega_0 = \partial \Omega_1$ be the interface between the two regions,
which is at least of class $C^2$.
Let $\varepsilon_i > 0$, $\omega > 0$, and $k_i = \omega \sqrt{\varepsilon_i} > 0$ be constants,
representing the material constant, angular frequency, and wavenumber, respectively, in each $\Omega_{i}$ for $i=0, 1$.
Let $u^{\mathrm{in}}$ be the incident wave field in $\Omega_0$
that is an analytic solution of the Helmholtz equation with wavenumber $k_0$ in $\mathbb{R}^d$.
Let us consider the following transmission problem:
find the scattered fields $u_j^{\mathrm{sc}} \in C^2(\Omega_j) \cap C^1({\overline \Omega_j})$ for $j = 0, 1$ such that
\begin{align}
  \begin{dcases}
    (\Delta + k_j^2) \, u_j^{\mathrm{sc}} (x) = 0, \quad x \in \Omega_j, \quad j = 0, 1, 
    \\
    u_0^{\mathrm{sc}} (x) + u^{\mathrm{in}} (x) = u_1^{\mathrm{sc}} (x) \quad (= u(x)), \quad x \in \Gamma,
    \\
    \frac{1}{\varepsilon_0} \pdv{(u_0^{\mathrm{sc}} + u^{\mathrm{in}})}{\nu}\qty(x)
    = \frac{1}{\varepsilon_1} \pdv{u_1^{\mathrm{sc}}}{\nu}\qty(x)
    \quad (= q (x)), \quad x \in \Gamma,
    \\
    u_0^{\mathrm{sc}}(x) \text{ satisfies the outgoing radiation condition},
  \end{dcases}
  \label{eq:problem}
\end{align}
where the unit normal vector $\nu (x)$ on $\Gamma$ is directed into $\Omega_0$ and $\pdv{}{\nu}$ stands for the normal derivative.
We note that the two functions $u, q$ on the boundary $\Gamma$ are defined in the boundary conditions of \eqref{eq:problem}.

For the prescribed analytic incident field, Lemma \ref{thm:recover_smoothness} and Corollary \ref{thm:solvability} imply that \(u\in C^{1,\beta}(\Gamma)\) and \(q\in C^{0,\beta}(\Gamma)\) for every \(0<\beta<1\).
In two dimensions, after choosing a regular $2\pi$-periodic parametrization of each boundary component,
the H\"older continuity of the pulled-back density ensures uniform convergence of its trigonometric interpolants on equispaced nodes \cite[Theorem 11.6]{kress2014linear}.
Thus, Hölder continuity or higher regularity is advantageous when solving the BIEs using Nyström methods.

The $C^2$ boundary assumption ensures the classical jump relations and mapping properties used below.
An extension to Lipschitz boundaries would require a Sobolev-space analysis using weak traces of $H^1$ solutions in the interior and solutions that belong to $H^1$ on bounded truncations of the exterior domain \cite{mclean2000strongly}.
Under such reduced regularity, careful attention must be paid to the function space to which the boundary solution belongs when applying numerical methods.
For instance, if the solution to the BIE belongs to $H^{1/2}(\Gamma) \times H^{-1/2}(\Gamma)$, it may lack pointwise meaning on the boundary.
Therefore, under the assumption of Lipschitz boundary regularity,
applying Galerkin methods is more natural than using Nyström methods.

\section{The weakly singular Burton--Miller equation}
For sufficiently regular densities $\varphi$,
we define the following boundary integral operators:
\begin{align}
  (S_k\varphi)(x) &= \int_{\Gamma} G_k(x,y)\varphi(y) \dd s(y), \quad 
  &&(D_k\varphi)(x) = \int_{\Gamma} \frac{\partial G_k (x,y)}{\partial \nu(y)}\varphi(y) \dd s(y),
  \\
  (D_k^* \varphi)(x) &= \int_{\Gamma} \frac{\partial G_k (x,y)}{\partial \nu(x)}\varphi(y) \dd s(y), \quad 
  &&(N_k\varphi)(x) = \frac{\partial}{\partial\nu(x)}\int_{\Gamma} \frac{\partial G_k(x,y)}{\partial \nu(y)}\varphi(y) \dd s(y),
\end{align}
where $G_k(x, y)$ is the fundamental solution of the Helmholtz equation with wavenumber $k$.
Let $C(\Gamma)$ be the space of continuous complex-valued functions on $\Gamma$ equipped with the maximum norm.

The standard BM equation for the problem \eqref{eq:problem} is given by
\begin{equation}
  \mqty(
  D_{k_0} - I/2 + i \eta N_{k_0} & - \varepsilon_0 S_{k_0} - i \eta \varepsilon_0 (D_{k_0}^* + I/2) \\
  D_{k_1} + I/2 & -  \varepsilon_1 S_{k_1}
  )
  \mqty(u \\ q)
  = 
  \mqty(
  -u^{\mathrm{in}} - i \eta \pdv{u^{\mathrm{in}}}{\nu} \\
  0
  ), \label{eq:bm}
\end{equation}
where $I$ is the identity operator,
$i$ is the imaginary unit, and $\eta > 0$ is a coupling constant. 
Recently, this equation has been demonstrated to be effective for numerical computations \cite{tomoyasu2026point, matsumoto2026fast}.
However, since it involves the hypersingular operator $N_{k_0}$, challenges remain in applying Nyström methods in three-dimensional problems.
To address this issue, the weakly singular BM equation, expressed as
\begin{equation}
  \mqty(
  D_{k_1} + I/2 & -\varepsilon_1 S_{k_1} 
  \\
  D_{k_0} - I/2 + i \eta \qty(N_{k_0} - N_{k_1}) & - \varepsilon_0 S_{k_0} -i \eta \qty(\varepsilon_0 D_{k_0}^* - \varepsilon_1 D_{k_1}^*) - \frac{i \eta (\varepsilon_0 + \varepsilon_1)}{2}I \\
  )
  \mqty(
  u \\
  q
  )
  =
  \mqty(
  0 \\
  -u^{\mathrm{in}} - i \eta \pdv{u^{\mathrm{in}}}{\nu}
  ), \label{eq:wsbm}
\end{equation}
is promising.
We can write this system compactly as 
\[
A \mqty(u \\ q) = b,
\]
where the linear operator $A : X \to X$ and the right-hand side $b \in X$ 
on the product space $X = C(\Gamma) \times C(\Gamma)$
are defined as
\begin{equation}
  A =
  \mqty(
  D_{k_1} + I/2 & -\varepsilon_1 S_{k_1} 
  \\
  D_{k_0} - I/2 + i \eta \qty(N_{k_0} - N_{k_1}) & - \varepsilon_0 S_{k_0} - i \eta \qty(\varepsilon_0 D_{k_0}^* - \varepsilon_1 D_{k_1}^*) - \frac{i \eta (\varepsilon_0 + \varepsilon_1)}{2}I \\
  ),
  \quad
  b = 
  \mqty(
  0 \\
  -u^{\mathrm{in}} - i \eta \pdv{u^{\mathrm{in}}}{\nu}
  ),
\end{equation}
respectively.

As preparation for the proof, we recall the following results.
\begin{theorem}[{\cite[Theorem 3.4]{colton2013inverse};
      \cite[Theorem 2.31]{colton1983integral}}] \label{thm:mapping}
  Let $k, k^\prime > 0$ with $k \neq k^\prime$ and $0 < \beta < 1$.
  Let the boundary $\Gamma = \partial \Omega_1$ be of class $C^2$.
  Then, the operators $S_k$, $D_k$, $D_k^*$, and $N_k - N_{k^\prime}$ are compact from $C(\Gamma)$ into $C(\Gamma)$.
  Furthermore, $S_k$, $D_k$, $D_k^*$, and $N_k - N_{k^\prime}$ are bounded from $C(\Gamma)$ into $C^{0,\beta}(\Gamma)$,
  and $S_k$ and $D_k$ are bounded from $C^{0,\beta}(\Gamma)$ into $C^{1,\beta}(\Gamma)$.
  The operator $N_k$ is bounded from $C^{1,\beta}(\Gamma)$ into $C^{0,\beta}(\Gamma)$.
\end{theorem}

\begin{theorem}[{\cite[Theorem 1.16]{colton1983integral}}] \label{thm:fred}
  Let $Y$ be a normed space, $B : Y \to Y$ a compact linear operator,
  and let $I - B$ be injective.
  Then the inverse operator $(I - B)^{-1}$ exists and is bounded.
\end{theorem}

Theorem \ref{thm:fred} is a special case of the Fredholm alternative.
This yields the following theorem.

\begin{theorem}[{\cite[Corollary 1.20]{colton1983integral}}] \label{thm:bdd}
  Let $Y$ be a normed space, $B : Y \to Y$ a compact linear operator,
  and $Z : Y \to Y$ a bounded linear operator that has a bounded inverse $Z^{-1}$.
  Let $Z - B$ be injective.
  Then, the inverse operator $(Z - B)^{-1}$ exists and is bounded.
\end{theorem}
\begin{proof}
  We can write $(Z - B) = Z(I - Z^{-1} B)$ using a bounded inverse of $Z$.
  The operator $Z^{-1} B : Y \to Y$ is compact.
  Since $(Z - B)$ is injective and $Z$ is invertible, $(I - Z^{-1}B)$ is also injective.
  Therefore, $(I - Z^{-1} B)$ has a bounded inverse by Theorem \ref{thm:fred}.
  Taking inverses in the identity $(Z - B) = Z(I - Z^{-1} B)$,
  we have $(Z - B)^{-1} = (I - Z^{-1} B)^{-1} Z^{-1}$,
  which implies that $(Z - B)^{-1}$ exists and is bounded.
\end{proof}

\section{Uniqueness and existence of the solution}
For density functions $\varphi, \psi \in C(\Gamma)$, we define the following potential:
\begin{align}
  U_{k_j}[\varphi, \psi](x) =  g(j) \qty{ \int_{\Gamma} \pdv{G_{k_j}(x, y)}{\nu(y)} \varphi(y) \dd s(y) - \varepsilon_j \int_{\Gamma} G_{k_j}(x, y) \psi(y) \dd s(y)}, \quad x \in \mathbb{R}^d \setminus \Gamma, \label{eq:potential}
\end{align}
where the function $g: \{ 0, 1 \} \to \mathbb{Z}$ is given by $g(0) = 1$ and $g(1) = -1$.
By Theorem \ref{thm:bdd}, to establish the well-posedness of the weakly singular BM equation \eqref{eq:wsbm},
it suffices to show that $A$ is injective and admits a decomposition $A = S - K$, where $S$ is a bounded linear operator with a bounded inverse and $K$ is a compact linear operator.

\begin{lemma} \label{thm:recover_smoothness}
  Assume that the conditions specified in Section \ref{sec:problem} hold, and let $0 < \beta < 1$.
  Let $(f, g)^\top \in C^{1, \beta}(\Gamma) \times C^{0, \beta}(\Gamma)$.
  Then any solution $(\varphi, \psi)^\top \in X$ of the weakly singular BM equation
  $
  A(\varphi, \psi)^\top = (f, g)^\top
  $
  belongs to $C^{1, \beta}(\Gamma) \times C^{0, \beta}(\Gamma)$.
\end{lemma}
\begin{proof}
  From the first and second rows of $A(\varphi, \psi)^\top = (f, g)^\top$,
  the following relations hold:
  \begin{align}
    \tfrac{1}{2} \varphi &= - D_{k_1} \varphi + \varepsilon_1 S_{k_1} \psi + f, \label{eq:1st} \\
    - \tfrac{1}{2}\varphi - \tfrac{i \eta(\varepsilon_0 + \varepsilon_1)}{2} \psi &= -D_{k_0}\varphi - i \eta \qty(N_{k_0} - N_{k_1})\varphi + \varepsilon_0 S_{k_0}\psi + i \eta \qty(\varepsilon_0 D_{k_0}^* - \varepsilon_1 D_{k_1}^*) \psi + g. \label{eq:2nd}
  \end{align}
  Substituting \eqref{eq:1st} into \eqref{eq:2nd}, we have 
  \begin{align}
    - \tfrac{i \eta(\varepsilon_0 + \varepsilon_1)}{2} \psi = -D_{k_0} \varphi - i \eta \qty(N_{k_0} - N_{k_1}) \varphi - D_{k_1} \varphi + \varepsilon_0 S_{k_0} \psi + i \eta \qty(\varepsilon_0 D_{k_0}^* - \varepsilon_1 D_{k_1}^*) \psi + \varepsilon_1 S_{k_1} \psi + f + g. \label{eq:3rd}
  \end{align}
  Since $(f + g) \in C^{0, \beta}(\Gamma)$,
  the mapping properties stated in Theorem \ref{thm:mapping}
  imply that the right-hand side of \eqref{eq:3rd} belongs to $C^{0, \beta}(\Gamma)$.
  Consequently, we have $\psi \in C^{0, \beta}(\Gamma)$, and from \eqref{eq:1st},
  we then deduce $\varphi \in C^{1, \beta}(\Gamma)$.
\end{proof}

\begin{lemma} \label{thm:injective}
  Under the assumptions in Section \ref{sec:problem},
  the linear operator $A$ is injective.
  In other words, for any $z = (\varphi, \psi)^\top$ with $\varphi, \psi \in C(\Gamma)$,
  the homogeneous equation $Az=0$, given by
  \begin{equation}
    \mqty(
    D_{k_1} + I/2 & -\varepsilon_1 S_{k_1} 
    \\
    D_{k_0} - I/2 + i \eta \qty(N_{k_0} - N_{k_1}) & - \varepsilon_0 S_{k_0} -i \eta \qty(\varepsilon_0 D_{k_0}^* - \varepsilon_1 D_{k_1}^*) - \frac{i \eta(\varepsilon_0 + \varepsilon_1)}{2}I \\
    )
    \mqty(
    \varphi \\
    \psi
    )
    =
    \mqty(
    0 \\
    0
    ), \label{eq:injec}
  \end{equation}
  has only the trivial solution.
\end{lemma}
\begin{proof}
  (i) Case $\varepsilon_0 = \varepsilon_1$:
  Substituting $\varepsilon_0 = \varepsilon_1$ in \eqref{eq:injec} yields the following equations:
  \begin{align}
    &D_{k_0} \varphi + (1/2) \varphi - \varepsilon_0 S_{k_0} \psi = 0, \label{eq:du-sq} \\
    &D_{k_0} \varphi - (1/2) \varphi - \varepsilon_0 S_{k_0} \psi - i \eta \varepsilon_0 \psi = 0.
  \end{align}
  Subtracting the first equation from the second equation, we obtain the relation $\varphi = - i \eta \varepsilon_0 \psi$.
  Substituting this into \eqref{eq:du-sq} yields
  $i \eta \varepsilon_0 D_{k_0} \psi + (i \eta \varepsilon_0 /2) \psi + \varepsilon_0 S_{k_0} \psi = 0$.
  Since $\varepsilon_0 > 0$, we have $(i \eta D_{k_0} + (i \eta /2) + S_{k_0}) \psi = 0$.
  From the uniqueness for the solution of the combined potential formulation with respect to a positive real $\omega > 0$ \cite{colton1983integral},
  we conclude that $\psi = 0$,
  which also implies $\varphi = 0$ since $\varphi = - i \eta \varepsilon_0 \psi$.

  (ii) Case $\varepsilon_0 \neq \varepsilon_1$.
  A minor adaptation of \cite[Lemma 3.1]{matsumoto2025efficient} establishes the injectivity of $A$. 
  Nevertheless, for the sake of continuity with Corollaries \ref{thm:solvability} and \ref{thm:eigenvalues}, we present the details here.
  For $x \in \Gamma$, define the boundary limits of a function $f$ by
  \[
  f|_{+} = \lim_{h \to 0_{+}} f(x + h\nu(x)), \quad f|_{-} = \lim_{h \to 0_{+}} f(x - h\nu(x)).
  \]
  Since the right-hand side $(0, 0)^\top \in C^{1, \beta}(\Gamma) \times C^{0, \beta}(\Gamma)$ for $0 < \beta < 1$,
  Lemma \ref{thm:recover_smoothness} implies that any solution $(\varphi, \psi)^\top \in X$ of \eqref{eq:injec} belongs to $C^{1, \beta}(\Gamma) \times C^{0, \beta}(\Gamma)$.
  We have the following jump relations:
  \begin{align}
    U_{k_1}[\varphi, \psi]|_{\pm} &= -\qty(D_{k_1} \pm \tfrac{1}{2}) \varphi(x) + \varepsilon_{1} S_{k_1} \psi(x), \quad x \in \Gamma, \label{eq:potw_trace} \\
    \eval{\pdv{U_{k_1}[\varphi, \psi]}{\nu}}_{\pm} &= -N_{k_1} \varphi(x) + \varepsilon_{1} \qty(D^{*}_{k_1} \mp \tfrac{1}{2}) \psi(x), \quad x \in \Gamma, \label{eq:potw_dev_trace}
  \end{align}
  where the notation $\eval{\pdv{}{\nu}}_{\pm}$ is interpreted in the sense of a limit.
  From \eqref{eq:injec}, we have $U_{k_1}[\varphi, \psi](x) |_{+} = 0$.
  Then, $U_{k_1}[\varphi, \psi](x)$ is identically zero for $x \in \Omega_0$
  since the homogeneous exterior Dirichlet problem of the Helmholtz equation
  \begin{align}
    \begin{dcases}
      \Delta U_{k_1}[\varphi, \psi] + k_{1}^{2} U_{k_1}[\varphi, \psi] = 0, \quad x \in \Omega_0, \\
      U_{k_1}[\varphi, \psi]|_{+} = 0, \quad x \in \Gamma, \\
      U_{k_1}[\varphi, \psi](x) \text{ satisfies the outgoing radiation condition},
    \end{dcases}
    \label{eq:bvp_dirichlet}
  \end{align}
  has only the trivial solution \cite[Theorem 3.21]{colton1983integral}.
  Therefore, the following conditions hold:
  \begin{gather}
    U_{k_1}[\varphi, \psi]|_{+} = \eval{\pdv{U_{k_1}[\varphi, \psi]}{\nu}}_{+} = 0. \label{eq:zero_w_trace}
  \end{gather}

  Similarly, we have the following jump relations with respect to the potential $U_{k_0}[\varphi, \psi]$:
  \begin{align}
    U_{k_0}[\varphi, \psi]|_{\pm} &= \qty(D_{k_0} \pm \tfrac{1}{2}) \varphi(x) - \varepsilon_0 S_{k_0} \psi(x), \quad x \in \Gamma, \label{eq:potv_trace} \\
    \eval{\pdv{U_{k_0}[\varphi, \psi]}{\nu}}_{\pm} &= N_{k_0} \varphi(x) - \varepsilon_0 \qty(D^{*}_{k_0} \mp \tfrac{1}{2}) \psi(x), \quad x \in \Gamma. \label{eq:potv_dev_trace}
  \end{align}
  From \eqref{eq:injec},
  the linear combination satisfies $U_{k_0}[\varphi, \psi]|_{-} + i \eta \eval{\pdv{U_{k_0}[\varphi, \psi]}{\nu}}_{-} + i \eta \eval{\pdv{U_{k_1}[\varphi, \psi]}{\nu}}_{+} = 0$.
  This can be reduced to $U_{k_0}[\varphi, \psi]|_{-} + i \eta \eval{\pdv{U_{k_0}[\varphi, \psi]}{\nu}}_{-} = 0$
  using \eqref{eq:zero_w_trace}.
  Then, 
  $U_{k_0}[\varphi, \psi](x)$ is identically zero for $x \in \Omega_1$
  since the homogeneous interior impedance problem of the Helmholtz equation
  \begin{align}
    \begin{dcases}
      \Delta U_{k_0}[\varphi, \psi] + k_{0}^{2} U_{k_0}[\varphi, \psi] = 0, \quad x \in \Omega_1, \\
      U_{k_0}[\varphi, \psi]|_{-} + i \eta \eval{\pdv{U_{k_0}[\varphi, \psi]}{\nu}}_{-} = 0, \quad x \in \Gamma,
    \end{dcases}
    \label{eq:bvp_impedance}
  \end{align}
  has only the trivial solution \cite[Theorem 2.3]{Chandler-Wilde_Graham_Langdon_Spence_2012}.
  We therefore obtain
  \begin{gather}
    U_{k_0}[\varphi, \psi]|_{-} = \eval{\pdv{U_{k_0}[\varphi, \psi]}{\nu}}_{-} = 0. \label{eq:zero_v_trace}
  \end{gather}

  Furthermore, $U_{k_0}[\varphi, \psi] = 0$ in $\Omega_0$ and $U_{k_1}[\varphi, \psi] = 0$ in $\Omega_1$ are established via
  the unique solvability of the transmission problems.
  We can see from \eqref{eq:potv_trace} and \eqref{eq:potv_dev_trace} that
  $ U_{k_0}[\varphi, \psi]|_{+} -  U_{k_0}[\varphi, \psi]|_{-} = \varphi$, 
  $\eval{\pdv{U_{k_0}[\varphi, \psi]}{\nu}}_{+} - \eval{\pdv{U_{k_0}[\varphi, \psi]}{\nu}}_{-} = \varepsilon_{0} \psi$.
  Similarly, we can see from \eqref{eq:potw_trace} and \eqref{eq:potw_dev_trace} that
  $U_{k_1}[\varphi, \psi]|_{+} -  U_{k_1}[\varphi, \psi]|_{-} = -\varphi$,
  $\eval{\pdv{U_{k_1}[\varphi, \psi]}{\nu}}_{+} - \eval{\pdv{U_{k_1}[\varphi, \psi]}{\nu}}_{-} = -\varepsilon_{1} \psi$.
  Applying \eqref{eq:zero_v_trace} and \eqref{eq:zero_w_trace} to these relations,
  we obtain the transmission conditions.
  In summary, we have the homogeneous transmission problem
  \begin{align}
    \begin{dcases}
      \Delta U_{k_0}[\varphi, \psi] + k_{0}^{2} U_{k_0}[\varphi, \psi] = 0, \quad x \in \Omega_0, \\
      \Delta U_{k_1}[\varphi, \psi] + k_{1}^{2} U_{k_1}[\varphi, \psi] = 0, \quad x \in \Omega_1, \\
      U_{k_0}[\varphi, \psi]|_{+} - U_{k_1}[\varphi, \psi]|_{-} = 0, \quad x \in \Gamma, \\
      \frac{1}{\varepsilon_0} \eval{\pdv{U_{k_0}[\varphi, \psi]}{\nu}}_{+} - \frac{1}{\varepsilon_1} \eval{\pdv{U_{k_1}[\varphi, \psi]}{\nu}}_{-} = 0, \quad x \in \Gamma, \\
      U_{k_0}[\varphi, \psi](x) \text{ satisfies the outgoing radiation condition},
    \end{dcases}
    \label{eq:trans_homo}
  \end{align}
  which has only the trivial solution $U_{k_0}[\varphi, \psi] = 0$ in $\Omega_0$ and $U_{k_1}[\varphi, \psi] = 0$ in $\Omega_1$ \cite{kress1977transmission}.
  Therefore, $\varphi = 0$ and $\psi = 0$.
\end{proof}

\begin{remark} \label{rem:dirivation}
  The derivation of the weakly singular BM equation \eqref{eq:wsbm} is based on the ``fictitious'' boundary value problems \eqref{eq:bvp_dirichlet} and \eqref{eq:bvp_impedance} in the proof of Lemma \ref{thm:injective}.
  Using the potentials $U_{k_0}$ and $U_{k_1}$,
  the weakly singular BM equation can be written as follows:
  \begin{equation}
    \begin{dcases}
      U_{k_1}[u, q]|_{+} = 0, \\
      U_{k_0}[u, q]|_{-} + i \eta \qty( \eval{\pdv{U_{k_0}[u, q]}{\nu}}_{-} + \eval{\pdv{U_{k_1}[u, q]}{\nu}}_{+}) = - u^{\mathrm{in}} - i \eta \tfrac{\partial u^{\mathrm{in}}}{\partial \nu}.
    \end{dcases}
  \end{equation}
  Similarly, the standard BM equation is expressed as follows:
  \begin{equation}
    \begin{dcases}
      U_{k_1}[u, q]|_{+} = 0, \\
      U_{k_0}[u, q]|_{-} + i \eta \eval{\pdv{U_{k_0}[u, q]}{\nu}}_{-} = - u^{\mathrm{in}} - i \eta \tfrac{\partial u^{\mathrm{in}}}{\partial \nu}.
    \end{dcases}
  \end{equation}
  Of these conditions for the standard BM equation,
  the upper one corresponds to the exterior Dirichlet problem \eqref{eq:bvp_dirichlet}, while the lower one corresponds to the interior impedance problem \eqref{eq:bvp_impedance}.
  An important feature common to these two BIEs is
  that they both explicitly incorporate the condition
  $U_{k_1}[u,q]\big|_{+}=0$.
  For positive real frequencies, this condition,
  together with the uniqueness of the exterior Dirichlet problem,
  implies that $\pdv{U_{k_1}[u,q]|_+}{\nu} = 0$.
  Adding the vanishing term $i \eta \pdv{U_{k_1}[u,q]|_+}{\nu}$ to the standard BM condition yields \eqref{eq:wsbm}, which cancels the leading singularity of $N_{k_0}$ by replacing $N_{k_0}$ with $N_{k_0} - N_{k_1}$.
\end{remark}

\begin{theorem}
  Under the assumptions in Section \ref{sec:problem}, the weakly singular BM equation \eqref{eq:wsbm} is well-posed.
\end{theorem}
\begin{proof}
  Let $S : X \to X$ be a bounded linear operator, and $K : X \to X$ be a compact linear operator such that $A = S - K$. Namely, $S$ and $K$ are expressed as
  \begin{equation}
    S = \mqty(
    I/2 & \\
    - I/2 & - \frac{i \eta(\varepsilon_0 + \varepsilon_1)}{2} I
    ), \quad
    K =
    - \mqty(
    D_{k_1}  & -\varepsilon_1 S_{k_1} 
    \\
    D_{k_0} + i \eta \qty(N_{k_0} - N_{k_1}) & - \varepsilon_0 S_{k_0} -i \eta \qty(\varepsilon_0 D_{k_0}^* - \varepsilon_1 D_{k_1}^*) \\
    ),
  \end{equation}
  where the compactness of $K$ follows from Theorem \ref{thm:mapping}.
  Clearly, $S$ is invertible and $S^{-1} : X \to X$ is given by
  \[
  S^{-1} = \mqty(
  2I & 0 \\
  - \frac{2}{i \eta(\varepsilon_0 + \varepsilon_1)} I & - \frac{2}{i \eta(\varepsilon_0 + \varepsilon_1)} I
  ).
  \]
  Since $S^{-1}$ is bounded and $S - K$ is injective by Lemma \ref{thm:injective}, $A^{-1} = (S - K)^{-1}$ exists and is bounded by Theorem \ref{thm:bdd},
  which implies that the weakly singular BM equation is well-posed.
\end{proof}

\begin{cor} \label{thm:solvability}
  Assume that the conditions specified in Section \ref{sec:problem} hold, and let $0 < \beta < 1$.
  Then, if $u$ and $q$ are the solutions of the weakly singular BM equation \eqref{eq:wsbm},
  $U_{k_0}[u, q]$ and $U_{k_1}[u, q]$ solve the transmission problem \eqref{eq:problem}.
\end{cor}
\begin{proof}
  By Lemma \ref{thm:recover_smoothness}, 
  any solution $u, q \in C(\Gamma)$ to \eqref{eq:wsbm} satisfies 
  $u \in C^{1, \beta}(\Gamma) \subset C^1(\Gamma)$ and $q \in C^{0, \beta}(\Gamma)$.
  Analogously to Lemma \ref{thm:injective},
  applying the potentials $U_{k_0}[u, q](x)$ and $U_{k_1}[u, q](x)$, defined by \eqref{eq:potential},
  to the weakly singular BM equation \eqref{eq:wsbm} yields
  $ U_{k_1}[u, q]|_{+} = \eval{\pdv{U_{k_1}[u, q]}{\nu}}_{+} = 0$
  and $\qty(U_{k_0}[u, q]|_{-} + u^{\mathrm{in}}) = \qty( \eval{\pdv{U_{k_0}[u, q]}{\nu}}_{-} + \pdv{u^{\mathrm{in}}}{\nu} ) = 0$.
  These relations follow from the unique solvability of the homogeneous exterior Dirichlet and homogeneous interior impedance problems.
  Combining these relations with the jump relations $ U_{k_0}[u, q]|_{+} - U_{k_0}[u, q]|_{-} = u$,
  $\eval{\pdv{U_{k_0}[u, q]}{\nu}}_{+} - \eval{\pdv{U_{k_0}[u, q]}{\nu}}_{-} = \varepsilon_{0} q$,
  $U_{k_1}[u, q]|_{+} -  U_{k_1}[u, q]|_{-} = -u$, and
  $\eval{\pdv{U_{k_1}[u, q]}{\nu}}_{+} - \eval{\pdv{U_{k_1}[u, q]}{\nu}}_{-} = -\varepsilon_{1} q$,
  we see that $U_{k_0}[u, q](x)$ and $U_{k_1}[u, q](x)$ solve the following transmission problem:
  \begin{align}
    \begin{dcases}
      (\Delta + k_{j}^{2} ) U_{k_j}[u, q] = 0, \quad x \in \Omega_j, \quad j = 0, 1, \\
      U_{k_0}[u, q]|_{+} - U_{k_1}[u, q]|_{-} = -u^{\mathrm{in}}, \quad x \in \Gamma, \\
      \frac{1}{\varepsilon_0} \eval{\pdv{U_{k_0}[u, q]}{\nu}}_{+} - \frac{1}{\varepsilon_1} \eval{\pdv{U_{k_1}[u, q]}{\nu}}_{-} = -\frac{1}{\varepsilon_0} \pdv{u^{\mathrm{in}}}{\nu}, \quad x \in \Gamma, \\
      U_{k_0}[u, q](x) \text{ satisfies the outgoing radiation condition}.
    \end{dcases}
  \end{align}
  The uniqueness of the solution of an inhomogeneous transmission problem \cite{kress1977transmission}
  ensures that this problem is equivalent to the original transmission problem \eqref{eq:problem}.
\end{proof}

We next characterize the fictitious eigenfrequencies of \eqref{eq:wsbm}.
Define $W$ by $W = \{ a \in \mathbb{C} \mid \Re(a) \ge 0, a \neq 0 \}$.
For complex frequencies $\omega \in W$, the wavenumbers are still defined by
$k_j=\omega\sqrt{\varepsilon_j}$, but are no longer restricted to
positive real values.
The parameters $\varepsilon_0$, $\varepsilon_1$,
and $\eta$ remain real and positive.
We note that, for complex frequencies, the radiation condition is interpreted in terms of outgoing waves rather than in the classical Sommerfeld sense.
For some wavenumber $k \in W$, an exterior solution is called outgoing if, outside
a sufficiently large ball, it admits an expansion in outgoing
cylindrical waves for $d=2$ or spherical waves for $d=3$,
with radial factors $H_n^{(1)}(k r)$ or $h_n^{(1)}(k r)$,
respectively, where $r=|x|$ for $x \in \mathbb{R}^d$.
For two outgoing solutions with the same wavenumber, angular orthogonality and cancellation of the radial factors imply that the bilinear boundary integral in Green’s second identity vanishes on a sufficiently large circle or sphere.

\begin{remark}
The mapping and compactness properties of the boundary integral
operators used below extend to all fixed wavenumbers $k\in W$.
Indeed, if $\operatorname{Im}k<0$, then $-k$ lies in the upper
half-plane and $(-k)^2=k^2$.
Hence, $G_k$ and $G_{-k}$ are fundamental
solutions of the same elliptic operator $\Delta+k^2$.
The difference between $G_{k}(x, y)$ and $G_{-k}(x, y)$ extends smoothly to $x=y$ by elliptic regularity \cite[Remark 1; \S 3.5]{Hormander1955theory}.
The differences
$S_k-S_{-k}$,
$D_k-D_{-k}$,
$D_k^* -D_{-k}^*$,
$N_k-N_{-k}$
are regular integral operators. Thus, the mapping and compactness
properties for $\operatorname{Im} k \ge 0$ in \cite[Theorems 2.23, 2.30, and 2.31; \S 3.10]{colton1983integral}
carry over to $k \in W$, as do the jump relations.

The Calder\'on identity $N_{k}S_{k} = D_{k}^{*}D_{k}^{*}-\frac{1}{4}I$ follows from the interior Green
representation and the jump relations, which are valid for these complex wavenumbers.
\end{remark}

\begin{cor} \label{thm:eigenvalues}
  Assume that the conditions specified in Section \ref{sec:problem} hold, and extend the parameter $\omega > 0$ to $\omega \in W$.
 Then, the weakly singular BM equation \eqref{eq:wsbm} has the same fictitious eigenvalues as the standard BM equation \eqref{eq:bm}. In particular, these fictitious eigenvalues with respect to $\omega$ coincide with the eigenvalues of the following boundary value problems:
  \begin{itemize}
  \item the homogeneous exterior Dirichlet problem \eqref{eq:bvp_dirichlet} (with $k_1 = \omega \sqrt{\varepsilon_1}$), and
  \item the homogeneous interior impedance problem \eqref{eq:bvp_impedance} (with $k_0 = \omega \sqrt{\varepsilon_0}$ and $i \eta$).
  \end{itemize}
\end{cor}
\begin{proof}
  Let $\Lambda_e$, $\Lambda_i$, and $\Lambda_t$ be the sets of frequencies $\omega \in W$ for which the homogeneous exterior Dirichlet problem \eqref{eq:bvp_dirichlet}, the homogeneous interior impedance problem \eqref{eq:bvp_impedance}, and the homogeneous transmission problem \eqref{eq:trans_homo} have nontrivial solutions, respectively.
  For $\varphi, \psi \in C(\Gamma)$,
  let $\sigma(A)$ denote the set of $\omega$ for which $A (\varphi, \psi)^\top = 0$ has nontrivial solutions.
  Here, Lemma \ref{thm:recover_smoothness} implies that $\varphi \in C^{1, \beta}(\Gamma)$ and $\psi \in C^{0, \beta}(\Gamma)$.
  Letting $E = \Lambda_t \setminus (\Lambda_e \cup \Lambda_i$),
  it suffices to show that $(\sigma(A) \setminus E) = (\sigma(A_0) \setminus E) = \Lambda_e \cup \Lambda_i$.
  Although $\Lambda_t$, $\Lambda_e$, and $\Lambda_i$ may contain overlapping eigenfrequencies,
  there is no need to distinguish such cases in this proof.

  We show that $(\sigma(A) \setminus E) =  \Lambda_e \cup \Lambda_i$ by confirming that $\sigma(A) = \Lambda_t \cup \Lambda_e \cup \Lambda_i$.
  If we assume the equation $A (\varphi, \psi)^\top = 0$ has a nontrivial solution for $\omega \notin \Lambda_t \cup \Lambda_e \cup \Lambda_i$, the same argument as in Lemma \ref{thm:injective} leads to a contradiction regarding the uniqueness of the homogeneous boundary value problems satisfied by the potentials $U_{k_0}[\varphi, \psi]$ and $U_{k_1}[\varphi, \psi]$.
  This means that $\sigma(A) \subset \Lambda_t \cup \Lambda_e \cup \Lambda_i$.
  Here, uniqueness of the homogeneous boundary value problems follows
  directly from $\omega \notin \Lambda_t \cup \Lambda_e \cup \Lambda_i$
  and the definitions of these sets.
  To show that $\Lambda_t \cup \Lambda_e \cup \Lambda_i \subset \sigma(A)$, we consider the following cases.
  \begin{itemize}
  \item (i) $\omega \in \Lambda_e$.
    There exists at least one nontrivial solution $w \in C^2(\Omega_0) \cap C^1({\overline \Omega_0})$ to the homogeneous exterior Dirichlet problem.
    We note that $w|_+ = 0$ from the boundary condition.
    For any $\tilde{\varphi} \in C^{1, \beta}(\Gamma)$, $\tilde{\psi} \in C^{0, \beta}(\Gamma)$, we write $(f, g)^\top = A (\tilde{\varphi}, \tilde{\psi})^\top$.
    Green's second identity in $\Omega_0$ for $w$ and $U_{k_1}[\tilde{\varphi}, \tilde{\psi}]$ results in
    \begin{align}
      0 &= \int_{\Gamma} \qty(w|_+ \eval{\pdv{U_{k_1}[\tilde{\varphi}, \tilde{\psi}]}{\nu}}_+) - \qty(\eval{\pdv{w}{\nu}}_+ U_{k_1}[\tilde{\varphi}, \tilde{\psi}]|_+) \dd s(x) \\
      &= - \int_{\Gamma} \qty(\eval{\pdv{w}{\nu}}_+ (-f)) \dd s(x)
      = \int_{\Gamma} \qty(\eval{\pdv{w}{\nu}}_+ f) \dd s(x).
    \end{align}
    By density and continuity, the identity
    $\int_{\Gamma} \eval{\pdv{w}{\nu}}_+ f \dd s(x) = 0$ with $f = -U_{k_1}[\tilde{\varphi}, \tilde{\psi}]|_+$
    also holds for all $\tilde{\varphi}, \tilde{\psi} \in C(\Gamma)$.
    Indeed, $C^{1,\beta}(\Gamma) \times C^{0,\beta}(\Gamma)$ is dense in
    $X$ with respect to the maximum norm, and
    the boundedness of $D_{k_1}$ and $S_{k_1}$ on $C(\Gamma)$ ensures
    that 
    the correspondence
    $(\tilde{\varphi}, \tilde{\psi})^\top \mapsto \int_{\Gamma} \qty(\eval{\pdv{w}{\nu}}_+ U_{k_1}[\tilde{\varphi}, \tilde{\psi}]|_+) \dd s(x)$
    defines a continuous linear functional on $X$.
    Since a bilinear form expressed as $\int_\Gamma s(x)t(x) \dd s(x)$
    for $s \in C(\Gamma)$ and $t \in C(\Gamma)$ is non-degenerate \cite[Theorem 4.4]{kress2014linear}, and $C^{1,\beta}(\Gamma)$ is dense in $C(\Gamma)$,
    there exists $\tilde{f} \in C^{1, \beta}(\Gamma)$ such that 
    $\int_{\Gamma} \qty(\eval{\pdv{w}{\nu}}_+ \tilde{f}) \dd s(x) \neq 0$ for $\eval{\pdv{w}{\nu}}_+ \neq 0$.
    This implies that $(\tilde{f}, 0)^\top$ is not in the range of $A$.
    Since $A$ is not surjective,
    the contrapositive of Theorem \ref{thm:bdd} shows that $A$ is not injective.
    Therefore, there exists a nontrivial solution for $A(\varphi, \psi)^\top = 0$, which means $\Lambda_e \subset \sigma(A)$.
  \item (ii) $\omega \in \Lambda_i$ and $\omega \notin \Lambda_e$.
    There exists at least one nontrivial solution $w \in C^2(\Omega_1) \cap C^1({\overline \Omega_1})$ to the homogeneous interior impedance problem.
    In particular, $w|_-+i\eta\, \pdv{w}{\nu}|_-=0$ on $\Gamma$ and $\eta > 0$.
    We note that $w|_- \neq 0$ since if $w|_- = \eval{\pdv{w}{\nu}}_- = 0$, the integral representation of $w$ in $\Omega_1$ is identically zero.
    Suppose, for contradiction, that $(0, \overline w|_-)^\top \in \operatorname{Ran}(A)$.
    Then there exists
    $(\tilde{\varphi}, \tilde{\psi})^\top \in X$ such that $A(\tilde{\varphi}, \tilde{\psi})^\top=(0, \overline w|_-)^\top$.
    Since $\overline w|_- \in C^1 (\Gamma) \subset C^{0,\beta}(\Gamma)$, Lemma \ref{thm:recover_smoothness} gives
    $ (\tilde{\varphi}, \tilde{\psi})^\top \in C^{1,\beta}(\Gamma)\times C^{0,\beta}(\Gamma)$.
    We note that both $w$ and $U_{k_0}[\tilde{\varphi}, \tilde{\psi}]$ satisfy the Helmholtz equation with the same wavenumber $k_0$ in $\Omega_1$.
    Then, Green's second identity in $\Omega_1$ for $w$ and $U_{k_0}[\tilde{\varphi}, \tilde{\psi}]$ results in
    \begin{align}
      0 &= \int_{\Gamma} \qty(w|_- \eval{\pdv{U_{k_0}[\tilde{\varphi}, \tilde{\psi}]}{\nu}}_-) - \qty(\eval{\pdv{w}{\nu}}_- U_{k_0}[\tilde{\varphi}, \tilde{\psi}]|_-) \dd s(x) \\
      &= \frac{1}{i \eta} \int_{\Gamma} w|_- \qty( i \eta \eval{\pdv{U_{k_0}[\tilde{\varphi}, \tilde{\psi}]}{\nu}}_- +  U_{k_0}[\tilde{\varphi}, \tilde{\psi}]|_-) \dd s(x),
    \end{align}
    where we use $\eval{\pdv{w}{\nu}}_- = - w|_-/(i \eta)$ from the boundary condition $w|_- + i \eta \eval{\pdv{w}{\nu}}_- = 0$.
    The first row of $A (\tilde{\varphi}, \tilde{\psi})^\top = (0, \overline w|_-)^\top$ yields $U_{k_1}[\tilde{\varphi}, \tilde{\psi}]|_+=0$.
    Since $\omega\notin\Lambda_e$, it follows that $U_{k_1}[\tilde{\varphi}, \tilde{\psi}]=0$ in $\Omega_0$,
    and hence $\eval{\pdv{U_{k_1}[\tilde{\varphi}, \tilde{\psi}]}{\nu}}_+=0$.
    Therefore, the second row of $A (\tilde{\varphi}, \tilde{\psi})^\top = (0, \overline w|_-)^\top$, that is
    \begin{align}
      U_{k_0}[\tilde{\varphi}, \tilde{\psi}]|_{-} + i \eta \eval{\pdv{U_{k_0}[\tilde{\varphi}, \tilde{\psi}]}{\nu}}_{-} + i \eta \eval{\pdv{U_{k_1}[\tilde{\varphi}, \tilde{\psi}]}{\nu}}_{+} = \overline w|_-,
    \end{align}
    reduces to
    \[
    U_{k_0}[\tilde{\varphi}, \tilde{\psi}]|_-+i\eta \eval{\pdv{U_{k_0}[\tilde{\varphi}, \tilde{\psi}]}{\nu}}_- = \overline w|_-.
    \]
    Substituting this relation into the preceding Green identity gives
    $0 = \frac{1}{i\eta}\int_\Gamma (w|_-) (\overline w|_-) \dd s(x) = \frac{1}{i\eta}\int_\Gamma \bigl| w|_- \bigr|^2 \dd s(x),$
    which contradicts $w|_- \neq 0$.
    Thus $(0, \overline w|_-)^\top\notin\operatorname{Ran}(A)$, and $A$ is not surjective.
    Theorem 3 implies that $A$ is not injective.
    Hence
    $
    \Lambda_i \setminus \Lambda_e \subset \sigma(A)
    $.
    Together with case (i), this yields $\Lambda_i \subset \sigma(A)$,
    including frequencies in $\Lambda_i\cap\Lambda_e$.
  \item (iii) $\omega \in \Lambda_t$ and $\omega \notin \Lambda_e \cup \Lambda_i$.
    Let $w_0 \in C^2(\Omega_0) \cap C^1({\overline \Omega_0})$
    and $w_1 \in C^2(\Omega_1) \cap C^1({\overline \Omega_1})$ be a nontrivial solution of the homogeneous
    transmission problem, with $w_0$ outgoing. Define
    $(\tilde{\varphi}, \tilde{\psi})\in C(\Gamma)\times C(\Gamma)$ by
    \[
    \tilde{\varphi} = w_0|_+ = w_1|_-,
    \quad
    \tilde{\psi} = \frac{1}{\varepsilon_0} \eval{\pdv{w_0}{\nu}}_+
    = \frac{1}{\varepsilon_1} \eval{\pdv{w_1}{\nu}}_-.
    \]
    The interior and outgoing exterior Green representation formulas give
    \[
    U_{k_j}[\tilde{\varphi}, \tilde{\psi}]
    =
    \begin{dcases}
      w_j & x \in \Omega_j,\\
      0   & x \in \Omega_{1-j},
    \end{dcases}
    \quad j=0,1.
    \]
    Taking limits from the complementary domains and combining the normal derivative of subtraction terms as in Remark \ref{rem:dirivation} yields
    $A(\tilde{\varphi}, \tilde{\psi})^\top = 0$.
    One can see that $(\tilde{\varphi}, \tilde{\psi})^\top \ne (0, 0)^\top$
    since otherwise these representations imply $w_0=w_1=0$.
    Thus $\omega \in \sigma(A)$.
    Applying this construction to every $\omega\in\Lambda_t$,
    we conclude that $\Lambda_t\subset\sigma(A)$.
  \end{itemize}
  These three cases imply that $\Lambda_t \cup \Lambda_e \cup \Lambda_i \subset \sigma(A)$.
  Therefore, $\sigma(A) = \Lambda_t \cup \Lambda_e \cup \Lambda_i$, and hence $(\sigma(A) \setminus E) =  \Lambda_e \cup \Lambda_i$.

  We denote by $A_0 : C^{1, \beta}(\Gamma) \times C^{0, \beta}(\Gamma) \to C^{0, \beta}(\Gamma) \times C^{1, \beta}(\Gamma)$ the linear operator
  associated with the standard BM equation \eqref{eq:bm}.
  Let $\sigma(A_{0})$ denote the set of $\omega$ for which $A_0 (\varphi, \psi)^\top = 0$ has nontrivial solutions.
  To prove that $\sigma(A_0) = \Lambda_t \cup \Lambda_e \cup \Lambda_i$,
  we first show that $A_0$ has index zero.
  Let $D : C^{0, \beta}(\Gamma) \times C^{1, \beta}(\Gamma) \to C^{1, \beta}(\Gamma) \times C^{0, \beta}(\Gamma)$ be an invertible linear operator defined by
  \begin{equation}
  D = \begin{pmatrix}
  & I \\
  I & -2 i \eta N_{k_1} \\
  \end{pmatrix}.
  \end{equation}
  We can calculate the product of $D$ and $A_0$ as
  \begin{align}
    D A_0 = \begin{pmatrix}
    D_{k_1} + I/2 & -\varepsilon_1 S_{k_1} \\
    D_{k_0} - I/2 + i\eta (N_{k_0} - N_{k_1}) - 2i \eta N_{k_1}D_{k_1} & -\varepsilon_0 S_{k_0} -i \eta \varepsilon_0 (D_{k_0}^* + I/2) + 2 i \eta \varepsilon_1 N_{k_1} S_{k_1}
    \end{pmatrix}.
  \end{align}
  Using the Calderón identity $N_{k_1}S_{k_1} = D_{k_1}^{*}D_{k_1}^{*}-\frac{1}{4}I$,
  we obtain
  \begin{align}
    DA_0 = \begin{pmatrix}
    I/2 & -\varepsilon_1 S_{k_1} \\
    & - \frac{i \eta (\varepsilon_0 + \varepsilon_1)}{2}I
    \end{pmatrix}
    + R,
  \end{align}
  where the upper triangular operator is bounded and invertible,
  and $R$ is compact on $C^{1, \beta}(\Gamma) \times C^{0, \beta}(\Gamma)$.
  The compactness of $R$ follows from the mapping properties of the boundary integral operators
  and a compact embedding $C^{1,\beta}(\Gamma)\hookrightarrow C^{0,\beta}(\Gamma)$ \cite[Lemma 6.36]{gilbarg1998elliptic}.
  Thus the index of $D A_0 : C^{1, \beta}(\Gamma) \times C^{0, \beta}(\Gamma) \to C^{1, \beta}(\Gamma) \times C^{0, \beta}(\Gamma)$ is zero.
  This also implies that the index of $A_0$ is zero since $D$ is invertible.
  For $\omega \in \Lambda_e \cup \Lambda_i$, arguments analogous to those
  in cases (i) and (ii) show that $A_0$ is not surjective.
  Since $A_0$ has index zero, we see that $A_0$ is not injective (for $\omega \in \Lambda_e \cup \Lambda_i$) since
  \[
  \dim \mathrm{Ker}(A_0) = \mathrm{codim} \, \mathrm{Ran}(A_0) > 0.
  \]
  From the construction in case (iii) and the same uniqueness
  argument for $\omega\notin\Lambda_t\cup\Lambda_e\cup\Lambda_i$,
  we obtain $\sigma(A_0) = \Lambda_t \cup \Lambda_e \cup \Lambda_i = \sigma(A)$.
  This concludes the proof, including the case
  $\varepsilon_0=\varepsilon_1$.
\end{proof}

\section{Discussion}
We summarize the main features of the weakly singular BM equation in comparison with other boundary integral formulations for Helmholtz transmission problems.
\begin{itemize}
\item Only weakly singular kernels:
  The weakly singular BM equation involves no kernels more singular than weakly singular ones,
  whereas the PMCHWT and standard BM equations involve a hypersingular kernel.
\item Fictitious eigenvalues:
  The fictitious eigenvalues of the weakly singular BM equation are associated with an exterior Dirichlet problem and an interior impedance problem, as in the standard BM equation.
  In contrast, those of the PMCHWT and M\"uller equations are associated with some transmission problems \cite{misawa2017boundary}.
\item Number of layer-potential terms:
  The weakly singular BM equation involves eight layer-potential terms, as do the PMCHWT and M\"uller equations,
  whereas the standard BM equation involves only six.
  A smaller number of layer-potential terms may lead to more efficient numerical computation \cite{matsumoto2026fast}.
\end{itemize}
For each fixed positive real frequency, the bounded invertibility of $A$ ensures continuous dependence of the boundary solution on the right-hand side.
This property, together with the absence of hypersingular kernels, is favorable for Nyström discretization.
Convergence and stability of a particular discretization are not established here.

The PMCHWT equation is convenient for mathematical analysis because of its close connection with the Calder\'on projector.
The second-kind structure of the Müller equation can be advantageous for iterative solution, although convergence also depends on the problem parameters and discretization.

A detailed investigation of the computational performance of the weakly singular BM equation is left for future research.
Another challenging but interesting direction is to identify conditions on the boundary geometry and coupling parameters under which the complex eigenvalues of the exterior Dirichlet problem and the interior impedance problem approach the real axis.

\section*{Acknowledgments}
We would like to thank Mr. Hiromochi Itoh and Dr. Kei Matsushima for their valuable comments.
This work was supported by JSPS KAKENHI grant number 24K20783, and by ``Joint Usage/Research Center for Interdisciplinary Large-scale Information Infrastructures (JHPCN)'' in Japan (Project ID: jh260048).
\bibliographystyle{elsarticle-num} 
\bibliography{ref_short}

@article{matsumoto2025efficient,
  title={Efficient {LU} factorization exploiting direct-indirect {B}urton-{M}iller equation for {H}elmholtz transmission problems},
  author={Matsumoto, Yasuhiro and Matsushima, Kei},
  journal={arXiv preprint arXiv:2512.14193},
doi={10.48550/arXiv.2512.14193},
  year={2025}
}

@book{colton2013inverse,
  title={Inverse acoustic and electromagnetic scattering theory (Third Edition)},
  author={Colton, David L and Kress, Rainer},
  year={2013},
  address={New York},
  doi={10.1007/978-1-4614-4942-3},
  publisher={Springer}
}

@book{colton1983integral,
  title={Integral Equation Methods in Scattering Theory},
  author={Colton, David L and Kress, Rainer},
  year={1983},
  address={New York},
  publisher={Wiley-Interscience Publication}
}

@book{kress2014linear,
  title={Linear Integral Equations (Third Edition)},
  author={Kress, Rainer},
  year={2014},
  address={New York},
doi={10.1007/978-1-4614-9593-2},
  publisher={Springer}
}

@article{kress1977transmission,
    author = {Kress, R. and Roach, G. F.},
    title = {Transmission problems for the {H}elmholtz equation},
    journal = {J. Math. Phys.},
    volume = {19},
    number = {6},
    pages = {1433-1437},
    year = {1978},
    month = {06},
    issn = {0022-2488},
    doi = {10.1063/1.523808},
}

@article{hiptmair2022spurious,
  title={Spurious quasi-resonances in boundary integral equations for the {H}elmholtz transmission problem},
  author={Hiptmair, Ralf and Moiola, Andrea and Spence, Euan A},
  journal={SIAM J. Appl. Math.},
  volume={82},
  number={4},
  pages={1446--1469},
  year={2022},
doi={10.1137/21M1447052},
  publisher={SIAM}
}

@article{CHEN1998529,
title = {On fictitious frequencies using dual series representation},
journal = {Mech. Res. Commun.},
volume = {25},
number = {5},
pages = {529-534},
year = {1998},
issn = {0093-6413},
doi = {10.1016/S0093-6413(98)00069-X},
author = {J.T. Chen}
}

@article{ROKHLIN1983257,
title = {Solution of acoustic scattering problems by means of second kind integral equations},
journal = {Wave Motion},
volume = {5},
number = {3},
pages = {257-272},
year = {1983},
issn = {0165-2125},
doi = {10.1016/0165-2125(83)90016-1},
author = {V. Rokhlin},
}

@book{muller1969foundations,
  title={Foundations of the Mathematical Theory of Electromagnetic Waves},
  author={Claus Müller},
  year={1969},
  address= "Heidelberg",
doi={10.1007/978-3-662-11773-6},
  publisher={Springer Berlin}
}

@book{chew2001fast,
  title={Fast and efficient algorithms in computational electromagnetics},
  Xauthor={Chew, Weng Cho and Michielssen, Eric Muscle, JM and Jin, Jian-Ming},
  author={Chew, Weng Cho and Michielssen, Eric and Song, JM and Jin, Jian-Ming},
  year={2001},
  address		= "USA",
  publisher={Artech House, Inc.}
}

@article{burton1971application,
  title={The application of integral equation methods to the numerical solution of some exterior boundary-value problems},
  author={Burton, A.J. and Miller, G.F.},
  journal={Proc. R. Soc. Lond. A},
  volume={323},
  number={1553},
  pages={201--210},
  year={1971},
  doi={10.1098/rspa.1971.0097},
  publisher={The Royal Society London}
}

@article{kleinman1988single,
  title={On single integral equations for the transmission problem of acoustics},
  author={Kleinman, RE and Martin, PA},
  journal={SIAM J. Appl. Math.},
  volume={48},
  number={2},
  pages={307--325},
  year={1988},
  publisher={SIAM},
doi={10.1137/0148016},
}

@article{stefanov2000resonances,
  title={Resonances near the real axis imply existence of quasimodes},
  author={Stefanov, Plamen},
  journal={C. R. Acad. Sci. Ser. I},
  volume={330},
  number={2},
  pages={105--108},
  year={2000},
doi={10.1016/S0764-4442(00)00105-1},
  publisher={Elsevier}
}

@article{misawa2017boundary,
  title={Boundary integral equations for calculating complex eigenvalues of transmission problems},
  author={Misawa, Ryota and Niino, Kazuki and Nishimura, Naoshi},
  journal={SIAM J. Appl. Math.},
  volume={77},
  number={2},
  pages={770--788},
  year={2017},
  doi={10.1137/16M1087436},
  publisher={SIAM}
}

@article{MISAWA201912-191201,
  title={A boundary integral equation suitable for the {N}ystr\"om method having the same complex eigenvalues as the {B}urton--{M}iller formulation for the {H}elmholtz equation in 2{D}},
  author={Ryota Misawa and Naoshi Nishimura},
  journal={Trans. Jpn. Soc. Comput. Methods Eng.},
  volume={19},
  number={ },
  pages={61-66},
  year={2019},
  doi={10.60443/jascome.12-191201},
note={in Japanese}
}

@article{tomoyasu2026point,
  title={Point {J}acobi-type preconditioning and parameter tuning for {C}alderon-preconditioned {B}urton--{M}iller method in transmission problems},
  author={Tomoyasu, Keigo and Isakari, Hiroshi},
  journal={Eng. Anal. Bound. Elem.},
  volume={183},
  pages={106572},
  year={2026},
doi={10.1016/j.enganabound.2025.106572},
  publisher={Elsevier}
}

@article{matsumoto2026fast,
  title={A fast direct solver for two-dimensional transmission problems of elastic waves},
  author={Matsumoto, Yasuhiro and Maruyama, Taizo},
  journal={Eng. Comput.},
  volume={42},
  pages={68},
  year={2026},
doi={10.1007/s00366-026-02302-8},
  publisher={Springer}
}

@article{Chandler-Wilde_Graham_Langdon_Spence_2012,
 title={Numerical-asymptotic boundary integral methods in high-frequency acoustic scattering},
 volume={21},
 DOI={10.1017/S0962492912000037},
 Xjournal={Acta Numerica},
journal={Acta Numer.},
 author={Chandler-Wilde, Simon N. and Graham, Ivan G. and Langdon, Stephen and Spence, Euan A.},
 year={2012},
 pages={89–305}
}

@book{mclean2000strongly,
  title={Strongly elliptic systems and boundary integral equations},
  author={McLean, William Charles Hector},
  year={2000},
  address={New York},
  publisher={Cambridge University Press}
}

@article{asakura2009numerical,
  title={A numerical method for nonlinear eigenvalue problems using contour integrals},
  author={Asakura, Junko and Sakurai, Tetsuya and Tadano, Hiroto and Ikegami, Tsutomu and Kimura, Kinji},
  journal={JSIAM Letters},
  volume={1},
  pages={52--55},
  year={2009},
doi={10.14495/jsiaml.1.52},
  publisher={The Japan Society for Industrial and Applied Mathematics}
}

@article{matsumoto2026proxy,
  title={Proxy-surface-based fast direct solver for {TE}-mode scattering problems on distributed memory systems},
  author={Matsumoto, Yasuhiro and Yokota, Rio},
  journal={arXiv preprint arXiv:2607.00790},
doi={10.48550/arXiv.2607.00790},
  year={2026}
}

@article{alpert1999hybrid,
  title={Hybrid {G}auss-trapezoidal quadrature rules},
  author={Alpert, Bradley K},
  journal={SIAM J. Sci. Comput.},
  volume={20},
  number={5},
  pages={1551--1584},
  year={1999},
doi={10.1137/S1064827597325141},
  publisher={SIAM}
}

@book{gilbarg1998elliptic,
  title={Elliptic partial differential equations of second order (Reprint of the 1998 Edition)},
  author={Gilbarg, David and Trudinger, Neil S},
  address={Berlin Heidelberg},
  year={2001},
doi={10.1007/978-3-642-61798-0},
  publisher={Springer}
}

@article{Hormander1955theory,
  title={On the theory of general partial differential operators},
  author={Hörmander, L.},
  journal={Acta Math.},
  volume={94},
  pages={161--248},
  year={1955},
doi={10.1007/BF02392492},
}

\end{document}